\documentclass[12pt]{amsart}
\usepackage{latexsym,color,amsmath,amsthm,amssymb,amscd,amsfonts}

\usepackage{hyperref}

\newcommand{\R}{\mathbb R}
\newcommand{\K}{\mathcal K}

\newcommand{\N}{\mathbb N}

\renewcommand{\L}{\mathcal L}

\newtheorem{thm}{Theorem}[section]

\newtheorem{cor}[thm]{Corollary}

\newtheorem{rmk}[thm]{Remark}

\begin{document}


\title[Hermite-Hadamard inequality and applications]{On a generalization of the Hermite-Hadamard inequality and applications in convex geometry}

\author[B. Gonz\'alez Merino]{Bernardo Gonz\'alez Merino}
\address{Departamento de Did\'actica de la Matem\'atica, Facultad de Educaci\'on, Universidad de Murcia, 30100-Murcia, Spain}\email{bgmerino@um.es}

\thanks{
This research is a result of the activity developed within the framework of the Programme in Support of Excellence Groups of the Regi\'on de Murcia, Spain, by Fundaci\'on S\'eneca, Science and Technology Agency of the Regi\'on de Murcia. Partially supported by Fundaci\'on S\'eneca project 19901/GERM/15, Spain, and by MICINN Project PGC2018-094215-B-I00 Spain.}

\date{\today}\maketitle

\begin{abstract}
In this paper we show the following result: if $C$ is an n-dimensional 0-symmetric convex compact set,
$f:C\rightarrow[0,\infty)$ is concave, and $\phi:[0,\infty)\rightarrow[0,\infty)$ is not identically zero, convex, with
$\phi(0)=0$, then
\[
\frac{1}{|C|}\int_C\phi(f(x))dx\leq \frac{1}{2}\int_{-1}^1\phi(f(0)(1+t))dt,
\]
where $|C|$ denotes the volume of $C$. If $\phi$ is strictly convex, equality holds
if and only if $f$ is affine, $C$ is a generalized symmetric cylinder and $f$ becomes $0$ at one of the basis
of $C$.

We exploit this inequality to answer a question of Francisco Santos on estimating the volume of a convex set
by means of the volume of a central section of it. Second, we also derive a corresponding estimate
for log-concave functions.
\end{abstract}
\date{\today}\maketitle

The classical Jensen's inequality \cite{J} states that
if $(X, \Sigma, \mu)$ is a probability space, then for any concave $f:\mathbb R\rightarrow\mathbb R$ and any $\mu$-integrable function $g:X \rightarrow\mathbb R$, we have that
\[
\int_{X}f(g(x))d\mu(x)\leq f\left(\int_{X}g(x)d\mu(x)\right),
\]
and moreover, equality holds if and only if either $f$ is affine or $g$ is independent of $x$.

Let $\K^n$ be the set of n-dimensional compact, convex sets.
A set $K\in\K^n$ is \emph{0-symmetric} if $K=-K$. Let $\K^n_0$ be the subset of $\K^n$ of 0-symmetric sets.
Notice that we will consistently use $C$ (resp.~$K$) in functional (resp.~geometric) inequalities.
For any set $K\in\K^n$, we denote by $|K|$ the \emph{volume} (or Lebesgue measure) of $K$ computed in its \emph{affine hull} $\mathrm{aff}(K)$, i.e.,
the smallest affine subspace containing $K$.
Let $\Vert x\Vert=\sqrt{x_1^2+\cdots+x_n^2}$ be the \emph{Euclidean norm} of $x=(x_1,\dots,x_n)\in\mathbb R^n$,
let $B^n_2=\{x\in\mathbb R^n:\Vert x\Vert\leq 1\}$ be the \emph{Euclidean unit ball} of $\mathbb R^n$, and
let $\omega_n=|B^n_2|$ be its volume. For every $x,y\in\mathbb R^n$, let $\langle x,y\rangle=x_1y_1+\cdots+x_ny_n$ be
the \emph{scalar product} of $x$ and $y$.
The \emph{center of mass} of $K\in\K^n$ is the point
\[
x_K=\frac{1}{|K|}\int_{K}xdx.
\]
A well-known consequence of Jensen's inequality is the following Hermite-Hadamard inequality:
for any $C\in\K^n$ and $f:C\rightarrow\R$ concave, then
\begin{equation}\label{eq:HH_Jensen}
\frac{1}{|C|}\int_Cf(x)dx\leq f(x_C),
\end{equation}
with equality sign if and only if $f$ is affine. It was named after Hermite 1881 and Hadamard 1893,
who proved independently \eqref{eq:HH_Jensen} in the 1-dimensional case. See \cite{DP} (and \cite{CalCar} or \cite{St})
and the references on it for other historical considerations and a comprehensive and complete view of this type of inequalities.

The mean value of $f$ measured in $C$ (the left-term in \eqref{eq:HH_Jensen}) has repeatedly appeared
during the development of different topics of Analysis and Geometry (cf.~\cite{HLP}).
For instance, Berwald \cite[Sect.~7]{Ber} studied monotonicity relations of $l_p$ means of concave functions over convex compact domains
(see also \cite{ABG} and \cite[Sect.~7]{AAGJV} for a translation of it).
Borell \cite{Bor} did a step further by showing some convexity relations in the same regard (Thms.~1 and 2).
See also Milman-Pajor \cite[2.6]{MP2} and \cite[Sect.~5]{GNT},
or the Hardy-Littlewood Maximal Function (cf.~\cite{Me}).

Let $\mathcal L^n_i$ be the set of \emph{i-dimensional linear subspaces} in $\mathbb R^n$.
For $K\in\mathcal K^n$ and $H\in\mathcal L^n_i$, let $P_HK$ be the \emph{orthogonal projection} of $K$ onto $H$.
Moreover, let $e_1,\dots,e_n$ be the vectors of the \emph{canonical basis} of $\mathbb R^n$. For every $A\subset\mathbb R^n$,
let $\mathrm{lin}(A)$ be the \emph{linear hull} of $A$, and let $A^\bot$ be the \emph{orthogonal subspace} to $A$, and let $\partial A$
be the \emph{boundary} of $A$.

In 2017 during the conference \emph{Convex, Discrete and Integral Geometry} \footnote{Bedlewo, Poland 2017 \url{http://bcc.impan.pl/17Convex/}}
Francisco Santos
asked the following question: What is the smallest
constant $c_n > 0$ such that
\[
|K|\leq c_n|K\cap e_1^\bot|
\]
for
every $K\in\mathcal K^n$ with $P_{\mathrm{lin}\{e_1\}}K=[-e_1,e_1]$.
One of our aims is to compute this constant $c_n$.
A similar inequality is derived in \cite[Lemma 5.2]{IVS} where it is used to bound the volume of empty lattice 4-simplices in terms
of volumes of 3-lattice polytopes.
Notice that for every $K\in\mathcal K^n$ and $H\in\mathcal L^n_i$, Fubini's theorem implies that
\[
|K|\leq |P_HK|\max_{x\in H^\bot}|K\cap (x+H^\bot)|.
\]
There exist special subspaces for which the inequality above strengthens.
In this regard Spingarn \cite{S} and later Milman and Pajor \cite{MP} proved that if $K\in\K^n$ and $H\in\L^n_i$, then
\begin{equation}\label{eq:volumes_and_centroid}
|K|\leq |P_HK||K\cap(x_K+H^\bot)|.
\end{equation}
It is even known the worst deviation between the maximal volume section and the one passing through the centroid of $K$
(cf.~\cite{Fr},\cite{MM}, and further extensions in \cite{SY}).
Surprisingly enough, a consequence of Jensen's inequality \eqref{eq:HH_Jensen}
shows that (cf.~Theorem \ref{thm:center_of_Projection} below) for every $K\in\mathcal K^n$ and $H\in\mathcal L^n_{n-1}$
then
\begin{equation}\label{eq:SectProjJensen}
|K|\leq |P_HK||K\cap(x_{P_HK}+H^\bot)|,
\end{equation}
and this choice can \emph{sometimes} be better than \eqref{eq:volumes_and_centroid} (cf.~Remark \ref{rmk:x_Pbetterx_C}).
In this regard, we prove the following result, extending the inequality above when $P_HK\in\mathcal K^n_0$
and answering to the question posed at the beginning. The result below can sometimes be better than
\eqref{eq:volumes_and_centroid}
up to a linear factor in the dimension of the subspace.
\begin{thm}\label{thm:0-symm_GeneralCase}
Let $K\in\K^n$ and $H\in\L^n_i$ be such that $P_HK=-P_HK$. Then
\[
|K|\leq\frac{2^{n-i}}{n-i+1}|P_HK||K\cap H^\bot|.
\]
If we assume w.l.o.g.~that $H=\R^i\times\{0\}^{n-i}$,
there is equality above if and only if there exist a $(n-i)\times i$ matrix $B_0$, and $u\in\R^i$
such that
\[
K\cap(x+H^\bot)=(x_1,B_0(x_1))+\lambda_x (K\cap H^\bot)\quad\text{where}\quad
\lambda_x=\frac{\langle u,x_1\rangle+|K\cap H^\bot|^\frac{1}{n-i}}{|K\cap H^\bot|^\frac{1}{n-i}},
\]
for every $x=(x_1,0)\in P_HK\subset\R^i\times\{0\}^{n-i}$.
If in addition $i\leq n-2$ then equality in the inequality implies that there exist
$(x_0,0)\in\R^i\times\R^{n-i}$ and $K_0\in\K^{i-1}$ such that $P_HK=[(-x_0,0),(x_0,0)]+(\{0\}\times K_0\times\{0\}^{n-i})$
and
\[
|K\cap(x+H^\bot)|=\left(1+\frac{\langle (x_0,0),x\rangle}{\Vert x_0\Vert^2}\right)|K\cap H^\bot|,
\]
for every $x\in P_HK$.
\end{thm}

Note that Theorem \ref{thm:0-symm_GeneralCase} solves the question of Francisco Santos with optimal constant $c_n=2^n/n$, also
characterizing the equality case.
At the core of the proof of Theorem \ref{thm:0-symm_GeneralCase} rests a generalization of \eqref{eq:HH_Jensen},
which is the main result of the paper.
\begin{thm}\label{thm:Ineq_Concave_func}
Let $C\in\K^n_0$, let $f:C\rightarrow[0,\infty)$ be concave, and let $\phi:[0,\infty)\rightarrow[0,\infty)$
be not identically zero, convex, such that $\phi(0)=0$. Then
\[
\frac{1}{|C|}\int_C\phi(f(x))dx\leq \frac{1}{2}\int_{-1}^1\phi(f(0)(1+t))dt.
\]
If $\phi$ is strictly convex, equality holds if and only if there after applying a suitable rotation
exist
$C_0\in\K^{n-1}_0$ and $x_0\in\R^n$ with $(x_0)_1>0$ such that
\[
C=[-x_0,x_0]+(\{0\}\times C_0)
\]
and such that $f$ is an affine function with $f(-x_0+x)=0$, for every $x\in\{0\}\times\R^{n-1}$.
\end{thm}

Notice that, due to the convexity of $\phi$, the term $\int_{-1}^1\phi(f(0)(1+t))dt/2$ in Theorem \ref{thm:Ineq_Concave_func}
is larger than $c\cdot\phi(f(0))$, for some constant $c\geq 1$, as in the case $\phi(t)=t^m$, $m\in\mathbb N$ (cf.~Corollary \ref{cor:tAlpha}).
Milman and Pajor (see \cite{MP}) proved that if $f:\R^n\rightarrow[0,\infty)$ is an integrable \emph{log-concave function}
(i.e.~$\log(f)$ is concave),
and $\mu:\R^n\rightarrow[0,\infty)$ is a probability measure, then
\begin{equation}\label{eq:MilmanPajor}
\int_{\R^n}f(x)d\mu(x)\leq f\left(\int_{\R^n}x\frac{f(x)}{\int_{\R^n}f(z)d\mu(z)}d\mu(x)\right),
\end{equation}
and equality holds if and only if $f(x)$ is independent of $x$. A direct consequence of this
result is the following Hermite-Hadamard inequality: for any $C\in\K^n$, $f:C\rightarrow[0,\infty)$ concave,
and $m\in\N$, then
\begin{equation}\label{eq:HHcenter_of_mass}
\frac{1}{|C|}\int_C f(x)^m dx\leq f(x_{f,m})^m,
\end{equation}
where $x_{f,m}=\int_{C}x\frac{f(x)^m}{\int_{\R^n}f(z)^mdz}dx$ ($f^m$ is log-concave if $f$ is concave).
Notice that \eqref{eq:HHcenter_of_mass} has a better constant than in Corollary \ref{cor:tAlpha}; however,
only in the latter the center is \emph{independent} of the function.

Using Theorem \ref{thm:Ineq_Concave_func} we also derive a Hermite-Hadamard inequality as in \eqref{eq:MilmanPajor}
evaluated at the center of mass of the domain. Notice that if $f(0)=f_{min}$, since $\lim_{a\rightarrow 1+}\frac{a^2-1}{\log (a^2)}=1$,
the right-term below becomes $f_{min}$.
\begin{thm}\label{thm:log-concaveDirect}
Let $C\in\mathcal K^n_0$ and let $f:C\rightarrow(0,\infty)$ be log-concave and continuous. Then
\[
\frac{1}{|C|}\int_Cf(x)dx\leq f_{min}\frac{(f(0)/f_{min})^2-1}{\log((f(0)/f_{min})^2)},
\]
where $f_{min}=\min\{f(x):x\in C\}$.
Equality holds if and only if there after applying a suitable rotation exist
$C_0\in\K^{n-1}_0$ and $x_0\in\R^n$ with $(x_0)_1>0$ such that
\[
C=[-x_0,x_0]+(\{0\}\times C_0)
\]
and $f=e^u$ where $u:C\rightarrow\R$ is an affine function with $u(-x_0+x)=\log(f_{min})$, for every $x\in\{0\}\times\R^{n-1}$.
\end{thm}
Notice that one could relax the continuity of $f$ above, by replacing it by $f(x)\geq c$ for some $c>0$.
However, log-concave functions are continuous in the interior of their domain,
and their integral is not affected by changes in the values attained in the boundary of its domain.

The \emph{Brunn-Minkowski inequality} states that for any $K_1,K_2\in\K^n$ and $\lambda\in[0,1]$, then
\begin{equation}\label{eq:BandM}
|(1-\lambda)K_1+\lambda K_2|^\frac1n\geq(1-\lambda)|K_1|^\frac1n+\lambda|K_2|^\frac1n,
\end{equation}
and equality holds if and only if $K_1$ and $K_2$ are dilates, or if they are lower dimensional
and contained in parallel hyperplanes (see \cite{Ga} and the references therein for an insightful and complete study of this inequality).

We split the proofs of the results into two sections. In Section \ref{sec:HH_ineqs} we prove the functional
inequality, i.e.~Theorems \ref{thm:Ineq_Concave_func} and \ref{thm:log-concaveDirect}.
Afterwards in Section \ref{sec:RS_ineqs} we show some volumetric
inequalities solving in particular the question of Francisco Santos posed above.

\section{Proof of the Orlicz-Jensen-Hermite-Hadamard type inequalities}\label{sec:HH_ineqs}

Let us start this section by remembering that the \emph{Schwarz symmetrization} of $K\in\K^n$
with respect to $\mathrm{lin}(u)$, $u\in\R^n\setminus\{0\}$, is the set
\[
\sigma_u(K)=\bigcup_{t\in\R}\left(tu+r_t(B^n_2\cap u^\bot)\right),
\]
where $r_t\geq 0$ is such that $|K\cap(tu+u^\bot)|=r_t^{n-1}\omega_{n-1}$.
It is well-known that $\sigma_u(K)\in\K^n$ and that $|\sigma_u(K)|=|K|$ (cf.~\cite[Section 9.3]{Gru} or \cite{Sch} for more details).
For every $K\in\mathcal K^n$ and $x\in\mathbb R^n\setminus\{0\}$, the \emph{support function} of $K$ at $x$
is defined by $h(K,x)=\sup\{\langle x,y\rangle:y\in K\}$.

\begin{proof}[Proof of Theorem \ref{thm:Ineq_Concave_func}]
Since $f$ is concave and non-negative in $C$, there
exists an affine function
$g:C\rightarrow[0,\infty)$
such that
\[
g(0)=f(0)\quad\text{and}\quad g(x)\geq f(x)\,\text{ for every }x\in C.
\]
Since $\phi$ is convex with $\phi(0)=0$, for any $x_2>x_1>0$ we have that
\[
0\leq\frac{\phi(x_1)-0}{x_1-0}\leq\frac{\phi(x_2)-\phi(x_1)}{x_2-x_1},
\]
i.e., $\phi$ is non-decreasing. Thus
\begin{equation}\label{eq:affine}
\int_C\phi(f(x))dx\leq \int_C\phi(g(x))dx.
\end{equation}
Now let $H:=\mathrm{aff}(G(g))$, where $G(g)=\{(x,g(x))\in C\times\mathbb R\}$ is the graph of $g$,
and observe that $H$ is an affine hyperplane in $\mathbb R^{n+1}$. Let us furthermore observe
that since $H\cap(g(0)e_{n+1}+\mathrm{lin}(\{e_1,\dots,e_n\}))\neq\emptyset$,
$\mathrm{dim}(H\cap(g(0)e_{n+1}+\mathrm{lin}(\{e_1,\dots,e_n\})))\geq n-1$, and there exists
$L\in\L^n_{n-1}$ such that
\[
g(0)e_{n+1}+(L\times\{0\})\subset H\cap(g(0)e_{n+1}+\mathrm{lin}(\{e_1,\dots,e_n\})).
\]
After a suitable rotation, we can assume that $L=\mathrm{lin}(\{e_2,\dots,e_n\})$, that $h(C,e_1)=t_0>0$, and that
\[
(t_0,(x_0)_2,\dots,(x_0)_n,g(0)+\delta)\in G(g),
\]
for some $(t_0,(x_0)_2,\dots,(x_0)_n)\in C$ and some $\delta\geq 0$.
Since $C$ is 0-symmetric and $g$ is affine, $(-t_0,-(x_0)_2,\dots,-(x_0)_{n},g(0)-\delta)\in G(g)$ too.
Observe that $g(0)-\delta\geq f((-t_0,-(x_0)_2,\dots,-(x_0)_{n}))\geq 0$, i.e., $\delta\leq g(0)$.

Observe also that $g$ is constant on each affine subspace $M_t=\{(t,x_2,\dots,x_n)\in C\}$, $t\in[-t_0,t_0]$.
Hence, if $(t,x_2,\dots,x_n)\in C$, let
\[
g(t,x_2,\dots,x_n)=g(0)+\frac{t}{t_0}\delta.
\]
Using Fubini's formula we have that
\[
\int_C \phi(g(x))dx=\int_{-t_0}^{t_0}\phi\left(g(0)+\frac{t\delta}{t_0}\right)|M_t|dt.
\]

Let us consider now $C':=\sigma_{e_1}(C)$.
If we denote by $M_t':=\{(t,x_2,\dots,x_n)\in C'\}$ for every $t\in[-t_0,t_0]$, then
\[
|M_t|=|M_t'|\quad\text{for every }t\in[-t_0,t_0]
\]
and in particular $|C|=|C'|$. Moreover, we also have that
$g(t,x_2,\dots,x_n)=g(0)+\frac{t}{t_0}\delta$ for every $(t,x_2,\dots,x_n)\in M_t'$.
Therefore
\[
\int_C\phi(g(x))dx=\int_{-t_0}^{t_0}\phi\left(g(0)+\frac{t\delta}{t_0}\right)|M'_t|dt.
\]
We now define the cylinders
\[
R_t:=(-te_1+M_t')+[-t_0e_1,t_0e_1]\quad\text{for every }t\in[0,t_0].
\]
We now prove that $R_{t_0}\subset C'\subset R_0$. For the left inclusion, since $C'$ is 0-symmetric
$-2t_0e_1+M'_{t_0}=(-t_0e_1+L)\cap C'\subset C'$ and $M_{t_0}'\subset C'$.
Then, the convexity of $C'$ yields $R_{t_0}=\mathrm{conv}((-2t_0e_1+M'_{t_0})\cup M'_{t_0})\subset C'$.
For the right inclusion, since $M_t'=(te_1+L)\cap C$ and $(-te_1+L)\cap C=-2te_1+M'_t$, $t\in[0,t_0]$, the convexity of $C'$
yields
\[
-te_1+M_t'=\frac12M_t'+\frac12(-2te_1+M'_t)\subset L\cap C'=M_0',
\]
from which $C'\subset M_0'+[-t_0e_1,t_0e_1]=R_0$, as desired.

Moreover, $(R_t)_t$ is a continuously decreasing family,
and thus there exists $t^*\in[0,t_0]$ such that $|R_{t^*}|=|C'|$. Let $R:=R_{t^*}$
and let $M_t'':=\{(t,x_2,\dots,x_n)\in R\}$ for $t\in[-t_0,t_0]$. Let us observe
that since $R$ and $C'$ are 0-symmetric and $M_t'$ and $M_t''$ are $(n-1)$-Euclidean balls centered
at $te_1$
\[
M_t'\subset M_t''\text{ if }|t|\in[t^*,t_0]\text{ and }M_t''\subset M_t'\text{ if }|t|\in[0,t^*].
\]
We also observe that $|C'|=|R|$ implies that $|C'\setminus R|=|R\setminus C'|$.
Let us furthermore denote by
\[
M_t^*:=M_t'\cap M_t''\quad\text{and}\quad M_t^{**}:=(M_t'\setminus M_t'')\cup(M_t''\setminus M_t').
\]
Then
\begin{equation}\label{eq:the_two_integrals}
\begin{split}
& \int_{-t_0}^{t_0}\phi\left(g(0)+\frac{t\delta}{t_0}\right)|M'_t|dt\\
& = \int_{-t_0}^{t_0}\phi\left(g(0)+\frac{t\delta}{t_0}\right)|M^*_t|dt+
\int_{-t^*}^{t^*}\phi\left(g(0)+\frac{t\delta}{t_0}\right)|M^{**}_t|dt.
\end{split}
\end{equation}

We start bounding from above the simpler left integral in \eqref{eq:the_two_integrals}, whose domain of integration is $C'\cap R$.
Let us observe that for every $a,r,\lambda\in\R$, $r\geq 0$, $\gamma\geq 1$, with $a-\gamma r\geq 0$,
we have that $\phi(a-r)+\phi(a+r)\leq\phi(a-\gamma r)+\phi(a+\gamma r)$.
In order to see this, remember that the definition of convexity of $\phi$ implies
\[
\frac{\phi(a+r)-\phi(a-r)}{2r},\frac{\phi(a+\gamma r)-\phi(a-\gamma r)}{2\gamma r}\leq\frac{\phi(a+\gamma r)-\phi(a-r)}{(\gamma+1)r},
\]
and denote by $m_1,m_3,m_2$ these slopes, respectively.
From this, if we define the lines
$y_1-(1/2)(\phi(a-r)+\phi(a+r))=m_1(x-a)$, $y_2-\phi(a-r)=m_2(x-r)$ and $y_3-(1/2)(\phi(a-\gamma r)+\phi(a+\gamma r))=m_3(x-a)$,
then we have that $y_1\leq y_2$ if $x\geq a-r$ and $y_2\leq y_3$ if $x\leq a+\gamma r$. In particular,
\begin{equation}\label{eq:convexFourFunc}
\frac12(\phi(a-r)+\phi(a+r))=y_1(a)\leq y_2(a)\leq y_3(a)=\frac12(\phi(a-\gamma r)+\phi(a+\gamma r)),
\end{equation}
as desired.

Since $\phi$ is a convex function, $\delta\rightarrow\phi(g(0)+(t/t_0)\delta)$ is convex too, $\delta\in[0,g(0)]$, using \eqref{eq:convexFourFunc}
with $a=g(0)$, $r=\delta t/t_0$ and $\gamma=g(0)/\delta$, we see that
\[
\begin{split}
& \int_{-t_0}^{t_0}\phi\left(g(0)+\frac{t\delta}{t_0}\right)|M^*_t|dt\\
& = \int_{0}^{t_0}\left(\phi\left(g(0)+\frac{t}{t_0}\delta\right)+\phi\left(g(0)-\frac{t}{t_0}\delta\right)\right)|M^*_t|dt\\
& \leq \int_{0}^{t_0}\left(\phi\left(g(0)+\frac{t}{t_0}g(0)\right)+\phi\left(g(0)-\frac{t}{t_0}g(0)\right)\right)|M^*_t|dt\\
& = \int_{-t_0}^{t_0}\phi\left(g(0)\left(1+\frac{t}{t_0}\right)\right)|M^*_t|dt.
\end{split}
\]

Now we focus in bounding from above the right integral in \eqref{eq:the_two_integrals}, whose domain is $(C'\setminus R)\cup(R\setminus C')$, partially using ideas from above. Using again that $\phi$ is convex, then $\delta\rightarrow\phi(g(0)+(t/t_0)\delta)$ is convex too, $\delta\in[0,g(0)]$.
Hence using \eqref{eq:convexFourFunc} yields
\begin{equation}\label{eq:delta}
\begin{split}
& \int_{-t^*}^{t^*}\left(g(0)+\frac{t\delta}{t_0}\right)|M^{**}_t|dt\\
&= \int_0^{t^*}\left(\phi\left(g(0)+\frac{t}{t_0}\delta\right)+\phi\left(g(0)-\frac{t}{t_0}\delta\right)\right)|M_t^{**}|dt\\
& \leq \int_0^{t^*}\left(\phi\left(g(0)\left(1+\frac{t}{t_0}\right)\right)+\phi\left(g(0)\left(1-\frac{t}{t_0}\right)\right)\right)|M_t^{**}|dt.
\end{split}
\end{equation}
Once more since $\phi$ is convex, $t\rightarrow\phi(g(0)(1+t/t_0))$ is convex too, together with \eqref{eq:convexFourFunc}, we get that
\begin{equation}\label{eq:CandR}
\begin{split}
& \int_0^{t^*}\left(\phi\left(g(0)\left(1+\frac{t}{t_0}\right)\right)+\phi\left(g(0)\left(1-\frac{t}{t_0}\right)\right)\right)|M_t^{**}|dt\\
& \leq \left(\phi\left(g(0)\left(1+\frac{t^*}{t_0}\right)\right)+\phi\left(g(0)\left(1-\frac{t^*}{t_0}\right)\right)\right)\int_0^{t^*}|M_t^{**}|dt \\
& = \left(\phi\left(g(0)\left(1+\frac{t^*}{t_0}\right)\right)+\phi\left(g(0)\left(1-\frac{t^*}{t_0}\right)\right)\right)\frac{|C'\setminus R|}{2} \\
& = \left(\phi\left(g(0)\left(1+\frac{t^*}{t_0}\right)\right)+\phi\left(g(0)\left(1-\frac{t^*}{t_0}\right)\right)\right)\frac{|R \setminus C'|}{2} \\
& = \left(\phi\left(g(0)\left(1+\frac{t^*}{t_0}\right)\right)+\phi\left(g(0)\left(1-\frac{t^*}{t_0}\right)\right)\right)\int_{t^*}^{t_0}|M_t^{**}|dt\\
& \leq \int_{t^*}^{t_0}\left(\phi\left(g(0)\left(1+\frac{t}{t_0}\right)\right)+\phi\left(g(0)\left(1-\frac{t}{t_0}\right)\right)\right)|M_t^{**}|dt \\
& = \int_{t^*}^{t_0}\phi\left(g(0)\left(1+\frac{t}{t_0}\right)\right)|M_t^{**}|dt + \int_{-t_0}^{-t^*}\phi\left(g(0)\left(1+\frac{t}{t_0}\right)\right)|M_t^{**}|dt.
\end{split}
\end{equation}

These two upper bounds prove from \eqref{eq:the_two_integrals} that
\[
\begin{split}
\int_C\phi(g(x))dx & = \int_{-t_0}^{t_0}\phi\left(g(0)+\frac{t\delta}{t_0}\right)|M_t|dt \\
 & \leq \int_{-t_0}^{t_0}\phi\left(g(0)\left(1+\frac{t}{t_0}\right)\right)|M^*_t|dt +
\int_{-t_0}^{-t^*}\phi\left(g(0)\left(1+\frac{t}{t_0}\right)\right)|M^{**}_t|dt \\
& + \int_{t^*}^{t_0}\phi\left(g(0)\left(1+\frac{t}{t_0}\right)\right)|M^{**}_t|dt = \int_R \phi(g_0(x))dx,
\end{split}
\]
where $g_0(x)$ is an affine function with $g_0(0)=g(0)$ and $g_0(-t_0,x_2,\dots,x_n)=0$ for every $(x_2,\dots,x_n)\in\R^{n-1}$.
Again by Fubini we now get that
\[
\begin{split}
\int_R \phi(g_0(x))dx & = \int_{-t_0}^{t_0}\phi\left(g(0)\left(1+\frac{t}{t_0}\right)\right)\frac{|R|}{2t_0}dt
= \frac12\int_{-1}^1\phi(g(0)(1+s))ds|R| \\
& = \frac12\int_{-1}^1\phi(f(0)(1+s))ds|C|,
\end{split}
\]
concluding the proof of the inequality.

Let us suppose that $\phi$ is strictly convex.
In the case of equality we must have equality in all inequalities above.
Let us notice that
the strict convexity of $\phi:[0,\infty)\rightarrow[0,\infty)$ implies strict monotonicity.
Indeed, if $x_2>x_1>0$, the strict convexity of $\phi$ implies that
\[
0\leq\frac{\phi(x_1)-0}{x_1-0}<\frac{\phi(x_2)-\phi(x_1)}{x_2-x_1},
\]
as desired.
Equality in \eqref{eq:affine} together with the strict monotonicity of $\phi$ implies that $f$ must be an affine function.
Equalities in \eqref{eq:CandR} together with the strict convexity of $\phi$ force that
$|M_t^{**}|=0$ for every $t\in[0,t_0]$, i.e., $C'$ has to fulfill
\[
|C'\cap(te_1+L)|=|R\cap(te_1+L)|=c
\]
for every $t\in[-t_0,t_0]$ and some constant $c>0$. Since $C'=\sigma_{e_1}C$,
we also have that
\[
|C\cap(te_1+L)|=|C'\cap(te_1+L)|=c.
\]
Notice that $(t_0-t)/(2t_0)\in[0,1]$, and thus by the convexity of $C$
\[
\left(1-\frac{t_0-t}{2t_0}\right)(C\cap(t_0e_1+L))+\frac{t_0-t}{2t_0}(C\cap(-t_0e_1+L))\subset C\cap(te_1+L).
\]
Then, the Brunn-Minkowski inequality \eqref{eq:BandM} implies that
\[
\begin{split}
c^\frac{1}{n-1} & =|C\cap(te_1+L)|^\frac{1}{n-1}\\
& \geq
\left|\left(1-\frac{t_0-t}{2t_0}\right)C\cap(t_0e_1+L)+\frac{t_0-t}{2t_0}C\cap(-t_0e_1+L)\right|^\frac{1}{n-1} \\
& \geq \left(1-\frac{t_0-t}{2t_0}\right)|C\cap(t_0e_1+L)|^{\frac{1}{n-1}}+\frac{t_0-t}{2t_0}|C\cap(-t_0e_1+L)|^\frac{1}{n-1}\\
& =c^{\frac{1}{n-1}}
\end{split}
\]
for every $t\in[-t_0,t_0]$.
Thus the equality case of Brunn-Minkowski inequality \eqref{eq:BandM} implies that $C\cap(t,x_2,\dots,x_n)$
is a translation of the same $(n-1)$-dimensional set for every $t\in[-t_0,t_0]$.
This is equivalent to the fact that
\[
C=[-x_0,x_0]+(\{0\}\times C_0),
\]
where $x_0=(t_0,(x_0)_2,\dots,(x_0)_n)$ and $C_0\in\K^{n-1}_0$. Finally, equality in \eqref{eq:delta}
forces that $\delta=f(0)$, i.e., that $g(-t_0,x_2,\dots,x_n)=0$
for every $(x_2,\dots,x_n)\in\R^{n-1}$, which concludes the equality case.
\end{proof}

\begin{rmk}
Notice that if $C_0\in\mathcal K^n_0$ then
\[
C=[-e_1,e_1]\times C_0\quad\text{and}\quad f(x)=\langle x,e_1\rangle-1
\]
attains equality in Theorem \ref{thm:Ineq_Concave_func} for every not identically zero, convex, non-decreasing function $\phi:[0,\infty)\rightarrow[0,\infty)$
with $\phi(0)=0$.
\end{rmk}

Our first corollary follows from applying Theorem \ref{thm:Ineq_Concave_func} to $\phi(t)=t^\alpha$.
In the next section we will use it to give new estimates of the volume of a convex body in terms
of the volumes of some of its sections and projections.
\begin{cor}\label{cor:tAlpha}
Let $C\in\mathcal K^n_0$, let $f:C\rightarrow[0,\infty)$ be concave, and let $\alpha\geq1$. Then
\[
\frac{1}{|C|}\int_Cf(x)^\alpha dx\leq \frac{2^\alpha}{\alpha+1}f(0)^\alpha.
\]
If $\alpha = 1$ equality holds if and only if $f$ is affine.
If $\alpha > 1$ equality holds if and only if $C$ is a
generalized cylinder, $C=[-x_0,x_0]+(\{0\}\times C_0)$, for some $C_0\in\mathcal K_0^{n-1}$ with $(x_0)_1>0$,
and $f$ is affine with $f(-x_0+x)=0$ for every $x\in\mathbb R^n$ with $x_1=0$.
\end{cor}

Yet another corollary to Theorem \ref{thm:Ineq_Concave_func} is when we apply it to $\phi(t)=e^t-1$.
\begin{proof}[Proof of Theorem \ref{thm:log-concaveDirect}]
Let $f(x)=e^{u(x)}$, with $u:C\rightarrow\R$ a concave function. Applying Theorem \ref{thm:Ineq_Concave_func}
to the function $u(\cdot)-u_0$ in $C$, where $u_0=\min_{x\in C}u(x)$, and to $\phi(t)=e^t-1$ we obtain that
\[
\begin{split}
\frac{1}{|C|}\int_Ce^{u(x)}dx & =\frac{e^{u_0}}{|C|}\int_C(e^{u(x)-u_0})dx \\
 & =e^{u_0}\left(1+\frac{1}{|C|}\int_C\phi(u(x)-u_0)dx\right) \\
 & \leq e^{u_0}\left(1+\frac12\int_{-1}^1\phi((u(0)-u_0)(1+t))dt \right)\\
 & =e^{u_0}\left(1+\frac12\int_{-1}^1(e^{(u(0)-u_0)(1+t)}-1)dt\right) \\
 & = e^{u_0}\left(1+\frac12\left(\frac{e^{2(u(0)-u_0)}-1}{u(0)-u_0}-2\right)\right) \\
 & =\frac{e^{u_0}}{2}\frac{e^{2(u(0)-u_0)}-1}{u(0)-u_0}
\end{split}
\]
which shows the result.

Since $\phi(t)=e^t-1$ is strictly convex,
the equality case follows immediately from the equality case of Theorem \ref{thm:Ineq_Concave_func}.
\end{proof}

\section{Estimating sizes of convex sets by their marginals}\label{sec:RS_ineqs}

We start this section by proving Theorem \ref{thm:0-symm_GeneralCase} as a consequence of Corollary \ref{cor:tAlpha}.
\begin{proof}[Proof of Theorem \ref{thm:0-symm_GeneralCase}]
By Fubini's formula, we have that
\[
|K|=\int_{P_HK}|K\cap(x+H^\bot)|dx.
\]
By the Brunn's Concavity Principle (see \cite[Prop.~1.2.1]{Giann}, see also \eqref{eq:BandM}) then
\[
f:H\rightarrow[0,\infty)\quad\text{where}\quad f(x):=|K\cap(x+H^\bot)|^{\frac{1}{n-i}}
\]
is a concave function.
After a suitable rigid motion, we assume that $H=\R^i\times\{0\}^{n-i}$. Corollary \ref{cor:tAlpha} then implies that
\[
\begin{split}
\int_{P_HK}f(x)^{n-i}dx & \leq\frac{2^{n-i}}{n-i+1}|P_HK|f(0)^{n-i}\\
& =\frac{2^{n-i}}{n-i+1}|P_HK||K\cap H^\bot|,
\end{split}
\]
concluding the result.

For the equality case, we must have equality in Corollary \ref{cor:tAlpha} where
$f(x)=|K\cap(x+H^\bot)|^\frac{1}{n-i}$, $C=P_HK$, and $\alpha=n-i$.
Hence, first of all, $f(x)$ must be an affine function. We hence can write
\[
f(x)=f(0)+\langle u,x\rangle,
\]
for some $u\in\R^i\times\{0\}^{n-i}$. This means in particular that
\[
\begin{split}
&|K\cap((1-\lambda)x+\lambda y)|^\frac{1}{n-i}\\
&=f((1-\lambda)x+\lambda y)\\
&=f(0)+\langle u,(1-\lambda)x+\lambda y\rangle\\
&=(1-\lambda)(f(0)+\langle u,x\rangle)+\lambda(f(0)+\langle u,y\rangle)\\
&=(1-\lambda)|K\cap(x+H^\bot)|^\frac{1}{n-i}+\lambda|K\cap(y+H^\bot)|^\frac{1}{n-i}.
\end{split}
\]
Hence, using Brunn-Minkowski equality case \eqref{eq:BandM},
we have that $K\cap(x+H^\bot)$ are dilates, of volume
\[
|K\cap(x+H^\bot)|^\frac{1}{n-i}=\langle u,x\rangle+|K\cap H^\bot|^\frac{1}{n-i}.
\]
Since $K$ is convex, there exists an $n\times n$ matrix $B$ of the form
\[
B=\left(\begin{array}{cc}\text{I}_i & 0\\ B_0 & 0\end{array}\right),
\]
where $B_0$ is an $(n-i)\times i$ matrix and $\text{I}_i$ is the $i$-dimensional identity matrix,
such that
\[
K\cap(x+H^\bot)=B(x)+\lambda_x (K\cap H^\bot),
\]
with
\[
\lambda_x=\frac{|K\cap(x+H^\bot)|^\frac{1}{n-i}}{|K\cap H^\bot|^\frac{1}{n-i}}
=\frac{\langle u,x\rangle+|K\cap H^\bot|^\frac{1}{n-i}}{|K\cap H^\bot|^\frac{1}{n-i}}.
\]
Second, if $\alpha=n-i\geq 2$, i.e.~$i\leq n-2$, then we moreover have that there exist $K_0\in\K^{i-1}$ and $x_0\in\R^i$
such that $P_HK=[-x_0,x_0]+(\{0\}\times K_0)$. Moreover, we must have also that
\[
f(-x_0,x_{i+1},\dots,x_n)=0\quad\text{for every }(x_{i+1},\dots,x_n)\in\R^{n-i},
\]
i.e., that $|K\cap((-x_0,x_{i+1},\dots,x_n)+H^\bot)|=0$. Once more since $f$ is affine, this means that
\[
u=\frac{|K\cap H^\bot|^\frac{1}{n-i}}{\Vert x_0\Vert^2}x_0,
\]
i.e., that
\[
|K\cap(x+H^\bot)|^\frac{1}{n-i}=\frac{|K\cap H^\bot|^\frac{1}{n-i}}{\Vert x_0\Vert^2}\langle x_0,x\rangle+|K\cap H^\bot|^\frac{1}{n-i},
\]
thus concluding the equality case.
\end{proof}

\begin{rmk}
For any $C_0\in\mathcal K^{i-1}_0$, $C_1\in\mathcal K^{n-i}$, the set
\[
C=\left\{(t,x_2,\dots,x_n):t\in[-1,1],(x_2,\dots,x_i)\in C_0,(x_{i+1},\dots,x_n)\in (1+t)C_1\right\}
\]
together with the subspace $H=\mathrm{lin}(\{e_1,\dots,e_i\})$ achieves equality in Theorem \ref{thm:0-symm_GeneralCase}.
\end{rmk}

We now properly state \eqref{eq:SectProjJensen} along with the characterization of its equality cases.
Notice that here we do not require $P_HK$ to be 0-symmetric.
\begin{thm}\label{thm:center_of_Projection}
Let $K\in\K^n$ and $H\in\L^n_{n-1}$. Then
\[
|K|\leq|P_HK||K\cap(x_{P_HK}+H^\bot)|.
\]
If we assume w.l.o.g.~that $H=\R^{n-1}\times\{0\}$ and that $x_{P_HK}=\{0\}$,
there is equality above if and only if there exist $b,u\in\R^{n-1}\times\{0\}$
such that $K\cap(x+H^\bot)=(x_1,\dots,x_{n-1},\langle b,x\rangle)+\lambda_x K\cap H^\bot$, where
\[
\lambda_x=\frac{\langle u,x\rangle+|K\cap H^\bot|}{|K\cap H^\bot|},
\]
for every $x\in P_HK$.
\end{thm}

\begin{proof}
Let us consider the function
\[
f:P_HK\rightarrow[0,\infty)\quad\text{with}\quad f(x)=|K\cap(x+H^\bot)|
\]
which by the Brunn's Concavity Principle, is a concave function.
Hence, using Fubini's formula and \eqref{eq:HH_Jensen} we directly obtain that
\[
|K|=\int_{P_HK}f(x)dx\leq|P_HK|f(x_{P_HK})=|P_HK||K\cap(x_{P_HK}+H^\bot)|,
\]
as desired.

The equality case follows as the equality case of Theorem \ref{thm:0-symm_GeneralCase}.
\end{proof}

We now give two observations out of Theorem \ref{thm:center_of_Projection}.
\begin{rmk}\label{rmk:x_Pbetterx_C}
Let us observe that Theorem \ref{thm:center_of_Projection} sometimes gives a tighter inequality than \eqref{eq:volumes_and_centroid}.
Indeed, if we consider the cone $K\in\mathcal K^n$ with apex at $e_n$ and basis $B^n_2\cap\mathrm{lin}(\{e_1,\dots,e_{n-1}\})$,
and consider $H=\mathrm{lin}(\{e_1,\dots,e_{n-2},e_n\})$, it is straightforward to check that
$P_HK=\mathrm{conv}((B^n_2\cap\mathrm{lin}(\{e_1,\dots,e_{n-2}\}))\cup\{e_n\})$, that
\[
x_K=\left(0,\dots,0,\frac{1}{n+1}\right)\quad\text{and}\quad x_{P_HK}=\left(0,\dots,0,\frac{1}{n}\right).
\]
Therefore, since $|K\cap(x_K+H^\bot)|=\frac{2n}{n+1}>\frac{2(n-1)}{n}=|K\cap(x_{P_HK}+H^\bot)|$,
\[
\frac{|K|}{|P_HK||K\cap(x_K+H^\bot)|}<\frac{|K|}{|P_HK||K\cap(x_{P_HK}+H^\bot)|}<1.
\]
\end{rmk}

One can combine two of those inequalities to show that any point in the line segment determined by two good choices
of points (as in \eqref{eq:volumes_and_centroid}), is again a good choice.
\begin{rmk}
If for some $K\in\K^n$ and $H\in\L^n_i$ there exist points $x_0,x_1\in K$ such that
\begin{equation}\label{eq:improved_estimate}
\frac{|K|}{|P_HK|}\leq|K\cap(x_j+H^\bot)|,\quad\text{for }j=0,1,
\end{equation}
then, for every $\lambda\in[0,1]$, the Brunn-Minkowski inequality \eqref{eq:BandM} gives that
\[
\begin{split}
|K\cap((1-\lambda)x_0 & +\lambda x_1+H^\bot)| \\
& \geq \left((1-\lambda)|K\cap(x_0+H^\bot)|^\frac{1}{n-i}+\lambda|K\cap(x_1+H^\bot)|^\frac{1}{n-i}\right)^{n-i}\\
& \geq \frac{|K|}{|P_HK|},
\end{split}
\]
i.e., all points $(1-\lambda)x_0+\lambda x_1$ also fulfills the inequality \eqref{eq:improved_estimate}, $\lambda\in[0,1]$.
In particular, from Theorem \ref{thm:center_of_Projection} and \eqref{eq:volumes_and_centroid} with $i=n-1$,
we obtain that $c_\lambda=(1-\lambda)x_K+\lambda x_{P_HK}$ gives also an inequality of the same type.
\end{rmk}



\emph{Acknowledgements:} I would like to thank David Alonso-Guti\'errez, Matthieu Fradelizi, Sasha Litvak, Kasia Wyczesany, Rafael Villa, and Vlad Yaskin
for useful remarks and discussions, and Francisco Santos for sharing the question that motivated this article.

I would also like to thank the invaluable comments and corrections of the anonymous referee that enormously improved the
presentation of this article.

\end{document}